\documentclass[notitlepage,11pt,reqno]{amsart} 
 \makeatletter
\@ifundefined{subjclassname@2020}{%
  \@namedef{subjclassname@2020}{\textup{2020} Mathematics Subject Classification}%
}{}
\makeatother
 
\usepackage{amsopn,esint,nicefrac}
\usepackage{mathabx}
\usepackage[final]{hyperref}
\usepackage[T1]{fontenc}
\usepackage{graphicx}
\usepackage{cancel,pdfsync}
\usepackage{bbm}
\usepackage{tikz}
\usetikzlibrary{patterns} 
\usepackage{float}

\usepackage[utf8]{inputenc}
\usepackage{apptools}
\usepackage[margin=1in]{geometry}
\allowdisplaybreaks
\newcommand{\stkout}[1]{\ifmmode\text{\sout{\ensuremath{#1}}}\else\sout{#1}\fi}

\renewcommand{\fint}{\mathop{\int\hskip -0,93em -\, \!\!\!}}

\newcommand{\R}{\mathbb{R}}

\newcommand{\Rn}{\mathbb{R}^n}
\usepackage{graphicx,enumitem,dsfont,upgreek}
 \usepackage[dvips]{epsfig}
 \usepackage[mathscr]{eucal}
\usepackage{amscd}
\usepackage{amssymb,nicefrac}
\usepackage{amsthm}
\usepackage{amsmath}
\usepackage{latexsym}
\usepackage{dsfont}
\usepackage{upref}
\usepackage{hyperref}

\usepackage{color}
\theoremstyle{plain}

\newtheorem{thm}{Theorem}[section]
\theoremstyle{plain}
\newtheorem{lem}[thm]{Lemma}

\newtheorem{example}{Example}[section]

\newtheorem{rem}{Remark}[section]
\theoremstyle{definition}
\newtheorem*{maintheorem*}{Main Theorem}
\newtheorem*{maincorollary*}{Main Corollary}
{%
\setcounter{enumi}{0}

\begin{enumerate}}%
{\end{enumerate} }

{%
\setcounter{enumi}{0}

\begin{enumerate}}%
{\end{enumerate} }

\newcommand{\norm}[1]{\ensuremath{\left\|#1\right\|}}

\newcommand{\dx}{\ensuremath{\, {\rm d}x}}
\newcommand{\dy}{\ensuremath{\, {\rm d}y}}

\newcommand{\dt}{\ensuremath{\, {\rm d}t}}

\def\e{{\text{e}}}

\numberwithin{equation}{section} \allowdisplaybreaks

\title[Coercivity and regularity with measurable kernels]{Coercivity and regularity for nonlocal divergence-form operators with measurable kernels}

\begin{document}

\author{Anup Biswas}

\address{Indian Institute of Science Education and Research-Pune, Dr.\ Homi Bhabha Road, Pashan, Pune 411008, INDIA.}
\email{anup@iiserpune.ac.in}

\begin{abstract}
We consider nonlocal divergence-form operators of the type
$${\rm PV}\int_{\Rn}(u(x)-u(y))K(x, y) \dy=0\quad \text{in} \; B_2,$$
where the kernel $K$ is symmetric ($K(x, y) = K(y, x)$). We first show that there exists a kernel satisfying the conditions of  \cite[Open Question 2.1]{FRRO24} for which this problem admits a bounded discontinuous weak solution, answering that question negatively. We then impose the stronger one-sided lower bound suggested in \cite{CS20,IS20}, together with an $L^p$ bound of $K$ over annuli, and prove a coercivity estimate and interior H\"older regularity of weak solutions. In particular, this establishes the coercivity conjecture of \cite{CS20,IS20} under the additional $L^p$ hypothesis for $p>1$.
\end{abstract}

\keywords{H\"{o}lder regularity, Harnack inequality, weak solutions, coercivity estimates}
\subjclass[2020]{Primary: 47G20, 35R09, Secondary: 35B65}

\maketitle


\section{Introduction and main results}
The main goal of this note is to answer an open problem of Fern\'andez-Real and Ros-Oton mentioned in  \cite[Open Question 2.1]{FRRO24}, \cite{FRRO24a}. To be precise, let $K:\Rn\times\Rn\to [0, \infty]$ be a measurable function
which is also symmetric, that is,
$$ K(x, y)=K(y, x) \quad \text{for all}\; x, y\in\Rn.$$
We consider an equation in the  divergence-form operator  given by
\begin{equation}\label{weak}
{\rm PV}\int_{\Rn}(u(x)-u(y))K(x, y) \dy=0\quad \text{in} \; B_2.
\end{equation}
The solution should be understood in the weak sense. To define the weak solution, we define the space
\begin{align*}
\mathbb{V}_{\rm loc}(\Omega|\Rn)=\Bigl\{w:\Rn\to\R\: :\; \int_{\Omega_1}\int_{\Rn} (w(x)-w(y))^2 K(x, y)\dx\,\dy\; <\infty \;\;  \mbox{for all $\Omega_1\Subset\Omega$}\Bigr\}.
\end{align*}
A function $u\in \mathbb{V}_{\rm loc}(B_2|\Rn)$ is said to be a weak solution to \eqref{weak} if 
\begin{equation}\label{eqnweak}
\int_{\Rn}\int_{\Rn} (u(x)-u(y))(\varphi(x)-\varphi(y)) K(x, y)\dx\, \dy= 0
\end{equation}
for all $\varphi\in C^\infty_c(B_2)$. A natural question is under what minimal conditions on $K$ weak solutions are locally H\"{o}lder continuous. To answer this question, we
begin with the following restriction on $K$ which can also be found in \cite[Open Question 2.1]{FRRO24}.
\begin{align}
\inf_{e\in\mathbb{S}^{n-1}}\int_{B_r(x)} (e\cdot (y-x))^2 K(x, y)\dy &\geq \lambda r^{2-2s}\quad \text{for all $r>0$ and $x\in \Rn$},\tag{L}\label{lower}
\\
\int_{B_{2r}(x)\setminus B_r(x)} K(x, y)\dy &\leq \Lambda r^{-2s}\quad \text{for all $r>0$ and $x\in \Rn$},\tag{$U_1$}\label{upper}
\end{align}
for some $s\in (0, 1)$ and $\Lambda, \lambda\in (0, \infty)$. The above conditions naturally arise when one considers the Fourier symbol of the above operators, at least when the kernel is translation-invariant. More precisely, when the 
kernel has no dependence on $x$, conditions \eqref{lower}-\eqref{upper} are known to be equivalent to $\mathcal{F}(\xi)\asymp |\xi|^{2s}$, where $\mathcal{F}$ denotes the Fourier symbol of the corresponding linear operator (see \cite[Chapter 2.2.1]{FRRO24}).
Regularity of the solutions, for the translation-invariant kernels satisfying \eqref{lower} and \eqref{upper}, has already been discussed in \cite{FRRO24,FRRO24a}. For more details on this divergence-form operator and a comprehensive literature
survey we refer the readers to the book \cite{FRRO24} and the introduction of \cite{CS20,FRRO24a}.

In \cite[Open Question 2.1]{FRRO24} it was asked if a weak solution to \eqref{weak} is locally H\"{o}lder continuous if $K$ satisfies \eqref{lower} and \eqref{upper}. As  Example~\ref{Example} shows,  a weak solution
to \eqref{weak} need not be continuous even if the kernel satisfies \eqref{lower} and \eqref{upper}. Therefore, one must consider a smaller class of kernels, to guarantee H\"older regularity of the solutions. From the work of 
Imbert and Silvestre \cite{IS20} we again know that  local H\"older continuity is implied by the  coercivity estimate, together with \eqref{upper}, given by
\begin{equation}\label{coer}
\int_{\Rn}\int_{\Rn} \frac{(w(x)-w(y))^2}{|x-y|^{n+2s}}\dx\, \dy\leq C \int_{\Rn}\int_{\Rn} (w(x)-w(y))^2 K(x, y)\dx\, \dy
\end{equation}
for all measurable functions $w:\Rn\to \R$. Interestingly, it is conjectured in Imbert--Silvestre \cite{IS20} and Chaker--Silvestre \cite{CS20} that such coercivity estimate holds for kernels
satisfying
\begin{equation}\tag{$L_+$}\label{new-lower}
\inf_{e\in\mathbb{S}^{n-1}}\int_{B_r(x)} (e\cdot (y-x))_+^2 K(x, y)\dy \geq \lambda r^{2-2s}\quad \text{for all $r>0$ and $x\in \Rn$}.
\end{equation}
The coercivity estimate plays an important role in non-cutoff Boltzmann collision operators and \cite[Assumption~1.1]{CS20} provided
a sufficient condition for this estimate. A key question that remained unsettled is whether this 
coercivity estimate holds under \eqref{new-lower}.

We note that \eqref{new-lower} is stronger than \eqref{lower}. Apart from the above hypothesis, we also require $K$ to satisfy a slightly stronger condition than \eqref{upper}. Denote by $A_r(x)=B_{2r}(x)\setminus B_r(x)$ the
annulus of width $r$ at a distance $r$ from $x$. We assume that for some $p\in (1, \infty]$ we have
\begin{equation}\label{new-upper}\tag{$U_p$}
\left[\frac{1}{|A_r(x)|}\int_{A_r(x)} (K(x, y))^p \dy\right]^{\frac{1}{p}}\leq \Lambda_p r^{-n-2s}\quad \text{for all $r>0$ and $x\in\Rn$}.
\end{equation}
It is easily seen that for $p=1$, \eqref{new-upper} is same as \eqref{upper} with $\Lambda=(2^n-1)|B_1|\Lambda_1$.
 Our main result answers the coercivity conjecture of \cite{CS20,IS20}
under this additional hypothesis \eqref{new-upper} and also establishes interior H\"older regularity of weak solutions.
\begin{thm}\label{main}
Suppose that \eqref{new-lower} and \eqref{new-upper} hold for some $p\in(1,\infty]$.
Then the coercivity estimate \eqref{coer} holds. Moreover, there exist
$\beta\in(0,1)$ and $C>0$, both depending only on $\lambda,\Lambda_p,s,n,p$, such that
every bounded solution $u\in\mathbb V_{\rm loc}(B_2|\Rn)$ of \eqref{weak} satisfies
\[
|u(x)-u(y)|\le C\,\norm{u}_{L^\infty(\Rn)}\,|x-y|^\beta\qquad\text{for all }x,y\in B_1 .
\]
\end{thm}

The proof rests on the verification of \cite[Assumption~1.1]{CS20}. A related condition
appears in the work of Schwab and Silvestre \cite{SS16} on the regularity of parabolic
nonlocal equations with rough kernels. The link between \eqref{new-lower}--%
\eqref{new-upper} and \cite[Assumption~1.1]{CS20} is provided by our key lemma,
Lemma~\ref{Klem}. Once that assumption is verified, the coercivity estimate follows from
\cite{CS20} and the H\"older estimate from \cite[Theorem~1.7]{DK20}.

In view of Theorem~\ref{main}, it is natural to ask whether H\"older regularity persists under \eqref{upper} and \eqref{new-lower} alone. Following the circulation of a draft of this paper, Fern\'andez-Real informed the author via personal communication that this exact question is raised as Open Question~2.1 in the updated version of \cite{FRRO26}. Furthermore, they have constructed a counterexample to \eqref{coer} and to regularity under \eqref{upper}--\eqref{new-lower} for $p=1$. Consequently, the conditions in Theorem~\ref{main} are optimal and cannot be reduced to $p=1$.

\section{Proofs of the main results}\label{Sec-results}
We begin with an example which will answer the open question in \cite{FRRO24,FRRO24a} negatively.
\begin{example}\label{Example}
Let, for $s\in (0, 1)$,
$$ K(x, y)=\frac{1}{|x-y|^{n+2s}}1_{\{x_1y_1\geq 0\}},$$
where $x_1$ and $y_1$ denote the first coordinate of $x$ and $y$, respectively. It is easy to see that $K(x, y)=K(y, x)$. Again, since $K(x, y)\leq |x-y|^{-n-2s}$,
\eqref{upper} holds with 
$$\Lambda=\frac{|\mathbb{S}^{n-1}|}{2s}\left(1-\frac{1}{4^s}\right).$$
In fact, \eqref{new-upper} is also satisfied for every $p\in (1, \infty]$. To verify \eqref{lower}, we consider $e\in \mathbb{S}^{n-1}$ and $r>0$.
Define
$$I_r(e)=\int_{B_r(0)} |e\cdot h|^2 |h|^{-n-2s} {\rm d}h.$$
Since the integrand is dominated by $|h|^{-n+2-2s}$, we have $I_r(e)<\infty$ for all $r>0$. Again, from the rotational invariance and scaling we obtain
$I_r(e)=r^{2-2s}I_1(e_1)$, where $e_1=(1, 0, \ldots, 0)$. Averaging over the standard basis we get
\begin{align*}
I_1(e_1)=\frac{1}{n}\int_{B_1} |h|^2 |h|^{-n-2s} {\rm d}h=\frac{|\mathbb{S}^{n-1}|}{n}\int_0^1 t^{1-2s}\dt=\frac{|\mathbb{S}^{n-1}|}{n(2-2s)}.
\end{align*}
Define $H^+=\{h_1\geq 0\}$ and $H^{-}=\{h_1\leq 0\}$.
Note that the reflection $h\mapsto -h$ takes $B_r\cap H^+$ to $B_r\cap H^-$. Furthermore, the quantity $(e\cdot h)^2 |h|^{-n-2s}$ remains unchanged under this reflection. Hence
\begin{equation}\label{P1}
\int_{B_r\cap H^{\pm}} (e\cdot h)^2 |h|^{-n-2s}\,{\rm d}h=\frac{1}{2}I_r(e)=\frac{|\mathbb{S}^{n-1}|}{2n(2-2s)} r^{2-2s}
\end{equation}
for all $e\in \mathbb{S}^{n-1}$.
Now fix $x\in \Rn$. If $x_1\leq 0$, then $\{x_1y_1\geq 0\}\supset \{y_1\leq x_1\}=\{y_1- x_1\leq 0\}$, implying
\begin{align*}
\int_{B_r(x)} (e\cdot (y-x))^2 K(x, y) \dy &\geq \int_{B_r(x)\cap\{y_1-x_1\leq 0\}} (e\cdot (y-x))^2 |x-y|^{-n-2s} \dy 
\\
&=\int_{B_r\cap H^{-}} (e\cdot h)^2 |h|^{-n-2s}\,{\rm d}h=\frac{|\mathbb{S}^{n-1}|}{2n(2-2s)} r^{2-2s},
\end{align*}
using \eqref{P1}.
A similar argument also holds when $x_1>0$. Thus, \eqref{lower} is satisfied with $\lambda=\frac{|\mathbb{S}^{n-1}|}{2n(2-2s)}$.

Now consider the bounded jump function $u(x)=1_{\{x_1>0\}}$. Note that
$$\int_{\Rn}\int_{\Rn} (u(x)-u(y))^2 K(x, y) \dx\, \dy=0,$$
since $|u(x)-u(y)| K(x, y)\neq 0$ only on the set $\{y_1= 0\}\cup\{x_1= 0\}$, which is of measure zero with respect to $\dx\dy$. Furthermore, $u$ also satisfies \eqref{eqnweak}, implying that $u$ is a weak solution to 
\eqref{weak}. Note that $u$ is discontinuous and therefore, answers the open question of \cite{FRRO24,FRRO24a} negatively.
\end{example}

\begin{rem}
For the kernel in Example~\ref{Example}, \eqref{new-lower} fails. To see it, consider $x=(x_1, x_2, \ldots, x_n)$ with $x_1<0$ and 
$e=e_1=(1, 0,\ldots, 0)$. Then $(e_1\cdot(y-x))^2_+K(x, y)$ is non-zero only for $x_1< y_1\leq 0$, a slab of width $|x_1|$.
Also, on this slab we have $(y_1-x_1)^2\leq x_1^2$, implying
\begin{align*}
\int_{B_r(x)} (e_1\cdot(y-x))^2_+K(x, y)\dy &\leq \int_{x_1< y_1\leq 0}\min\{x^2_1, |x-y|^2\}|x-y|^{-n-2s}\dy
\\
&\leq \int_{|x-y|\leq |x_1|}|x-y|^{-n+2-2s}\dy + |x_1|^2\int_{|x-y|> |x_1|}|x-y|^{-n-2s}\dy
\\
&\lesssim |x_1|^{2-2s}\rightarrow 0,
\end{align*}
as $x_1\to 0$. Thus \eqref{new-lower} fails.
\end{rem}

To prove Theorem~\ref{main} we need the following key lemma.
\begin{lem}\label{Klem}
Suppose that \eqref{new-lower} and \eqref{new-upper} are satisfied for some $p\in (1, \infty]$. Then there are constants $\kappa, \lambda^\prime, \mu\in (0, 1)$, depending only on
$n,s,p,\lambda,\Lambda_p$, such that for all $x\in\Rn$, $r>0$ and $|e|=1$, the set
$$
N_e(x,r):=\left\{y\in B_r(x)\setminus B_{\kappa r}(x)\; :\; K(x, y)\geq \lambda^\prime|x-y|^{-n-2s}, \quad \e\cdot (y-x)\geq \kappa |y-x|\right\}
$$
satisfies $|N_e(x, r)|\geq \mu |B_r|$.
\end{lem}

\begin{proof}
We cover $B_r(x)$ with (not necessarily disjoint) four sets as follows: $B_r(x)=B_{\kappa r}(x) \cup D_1\cup D_2\cup N_e(x, r)$,  where 
\begin{align*}
D_1&:=\{y\in B_r(x)\setminus B_{\kappa r}(x)\; :\; K(x, y)< \lambda^\prime|x-y|^{-n-2s}\},
\\
D_2&:=\{y\in B_r(x)\setminus B_{\kappa r}(x)\; :\; (\e\cdot (y-x))_+< \kappa |y-x|\}.
\end{align*}
Using \eqref{new-upper}, we first calculate
\begin{align}\label{P2}
\int_{B_{\kappa r}(x)} (e\cdot (y-x))^2_+ K(x, y) \dy &=\sum_{i=0}^\infty \int_{B_{\frac{\kappa r}{2^{i}}}(x)\setminus B_{\frac{\kappa r}{2^{i+1}}}(x)} (e\cdot (y-x))^2_+ K(x, y) \dy\nonumber
\\
&\leq \sum_{i=0}^\infty (\kappa r 2^{-i})^2 |A_{\frac{\kappa r}{2^{i+1}}}(x)|  \fint_{B_{\frac{\kappa r}{2^{i}}}(x)\setminus B_{\frac{\kappa r}{2^{i+1}}}(x)} K(x, y) \dy\nonumber
\\
&\leq \sum_{i=0}^\infty (\kappa r 2^{-i})^2 |A_{\frac{\kappa r}{2^{i+1}}}(x)|  
\left[\fint_{B_{\frac{\kappa r}{2^{i}}}(x)\setminus B_{\frac{\kappa r}{2^{i+1}}}(x)} (K(x, y))^p \dy\right]^{\frac{1}{p}}\nonumber
\\
&\leq (\kappa r)^{2-2s} C_n \Lambda_p\sum_{i=0}^\infty (2^{-i})^{2-2s}\nonumber
\\
&\leq C_{n,s} \Lambda_p (\kappa r)^{2-2s}
\end{align}
where the constant $C_{n,s}$ depends only on $n$ and $s$. Next we compute the contribution from $D_1$.
\begin{align}\label{P3}
\int_{D_1} (e\cdot (y-x))^2_+ K(x, y)\dy\leq \lambda^\prime \int_{B_r(x)} |x-y|^{2-n-2s}\dy= \lambda^\prime \tilde{C}_{n,s} r^{2-2s}
\end{align}
for some constant $\tilde{C}_{n,s}$, depending only on $n$ and $s$. Again, from the calculation of \eqref{P2}, we get
\begin{equation}\label{P4}
\int_{D_2} (e\cdot (y-x))^2_+ K(x, y)\dy\leq \kappa^2 \int_{B_r(x)}|x-y|^{2} K(x, y)\dy\leq \kappa^2 {C}_{n,s}\Lambda_p r^{2-2s}.
\end{equation}
We set $\kappa\in (0, \frac{1}{2})$ so that $\kappa^{2-2s}\leq \frac{\lambda}{4 C_{n,s}\Lambda_p}$. This also implies 
$\kappa^{2}\leq \frac{\lambda}{4 C_{n,s}\Lambda_p}$. Also, let $\lambda^\prime=\frac{\lambda}{4 \tilde{C}_{n,s}}\wedge \frac{1}{2}$.
Thus, from \eqref{P2}--\eqref{P4}, we obtain
$$
\int_{B_{\kappa r}(x)\cup D_1\cup D_2} (e\cdot(y-x))^2_+K(x, y)\dy\leq \frac{3\lambda}{4}r^{2-2s}.
$$
From \eqref{new-lower} we thus have
\begin{equation}\label{P5}
\int_{N_e(x,r)} (e\cdot(y-x))^2_+K(x, y)\dy\geq \frac{\lambda}{4}r^{2-2s}.
\end{equation}
In what follows, the fact $p>1$ becomes crucial. We cover $B_{r}(x)\setminus B_{\kappa r}(x)$ with the annuli $A_{\rho_j}(x)$, where
$\rho_j=2^{-j-1} r$ for $j=0, \ldots, J$ and $J=\lceil \log_2(\frac{1}{\kappa})\rceil-1$. Note that $\rho_J\geq \kappa r/2$. Using
\eqref{P5}, \eqref{new-upper} and H\"older inequality we get
\begin{align*}
\frac{\lambda}{4}r^{2-2s}\leq r^2 \int_{N_e(x,r)} K(x, y)\dy &\leq r^2 \sum_{j=0}^{J}\int_{N_e(x,r)\cap A_{\rho_j}(x)} K(x, y)\dy 
\\
&\leq r^2 \sum_{j=0}^{J} \norm{K(x,\cdot)}_{L^p(A_{\rho_j}(x))} |N_e(x, r)\cap A_{\rho_j}(x)|^{\frac{p-1}{p}}
\\
& \leq r^2 \Lambda_p\left(\frac{\kappa r}{2}\right)^{-n-2s} C_{n,p}\, r^{\frac{n}{p}}\sum_{j=0}^{J}  |N_e(x, r)\cap A_{\rho_j}(x)|^{\frac{p-1}{p}}
\\
&\leq \frac{2^{n+2s}\Lambda_p C_{n,p}}{\kappa^{n+2s}} r^{-\frac{p-1}{p}n} r^{2-2s} (J+1)^{\frac{1}{p}}|N_e(x, r)|^{\frac{p-1}{p}},
\end{align*}
where in the last line we used concavity of $[0, \infty)\ni t\mapsto t^{\frac{p-1}{p}}$. Therefore, setting
$$\mu =\frac{1}{|B_1|} \left(\frac{\lambda \kappa^{n+2s}}{4\, 2^{n+2s}\Lambda_p C_{n,p}(J+1)^{\frac{1}{p}}}\right)^{\frac{p}{p-1}},$$
we conclude that $|N_e(x, r)|\geq \mu |B_r|$. This concludes the proof.
\end{proof}

Next using Lemma~\ref{Klem} we show that our kernel satisfies \cite[Assumption~1.1]{CS20}, and therefore, coercivity estimate follows.
\begin{lem}\label{L2.2}
Suppose that \eqref{new-lower} and \eqref{new-upper} are satisfied for some $p\in (1, \infty]$. Let $\kappa, \lambda^\prime, \mu$ be the
same as in Lemma~\ref{Klem}. Then for every ball $B=B_R(z)$ and every $x\in B$ we have
\begin{equation}\label{CS-cond}
|\{y\in B\; : \; K(x, y)\geq \lambda^\prime|x-y|^{-n-2s}\}|\geq \mu_{CS} |B|,\quad \text{where}\quad \mu_{CS}:=\mu \kappa^n 2^{-n}.
\end{equation}
Moreover, $K$ satisfies the coercivity estimate \eqref{coer} and there exists a constant $C$, dependent only on $n, s, p, \lambda, \Lambda_p$,
such that for any ball $B$ we have
\begin{equation}\label{ball-coer}
\int_{B}\int_{B} \frac{(w(x)-w(y))^2}{|x-y|^{n+2s}}\dx\, \dy\leq C \int_{B}\int_{B} (w(x)-w(y))^2 K(x, y)\dx\, \dy
\end{equation}
for all measurable functions $w$.
\end{lem}

\begin{proof}
Denote by
$$ D=\{y\in B\; : \; K(x, y)\geq \lambda^\prime|x-y|^{-n-2s}\}.$$
First, suppose $|x-z|\leq \frac{R}{2}$. Then $B_{\frac{R}{2}}(x)\subset B$, and we can apply Lemma~\ref{Klem} with $r=\frac{R}{2}$ and any $e\in\mathbb{S}^{n-1}$, giving us
$$|D|\geq |N_e(x,r)|\geq \mu |B_{r}(x)|=\mu 2^{-n}|B|.$$
Now assume $|x-z|> \frac{R}{2}$. Let $e=\frac{z-x}{|z-x|}$ and $r=\kappa R$ where $\kappa$ is given by Lemma~\ref{Klem}. For $y\in N_e(x, r)$ we define $y=x+t v$ with $|v|=1$.
Note that $t=|x-y|>0$. Since
$y\in B_r(x)\setminus B_{\kappa r}(x)$, we get $0<t< r=\kappa R$. Also, 
$(e\cdot (y-x))_+\geq \kappa |y-x|$ implies  $(e\cdot v)_+=e\cdot v \geq \kappa$. Hence
$$|y-z|^2=|x-z|^2-2t|x-z| e\cdot v + t^2\leq |x-z|^2 - t(2\kappa|x-z|-t)<|x-z|^2<R^2,$$
since $t<\kappa R<2\kappa|x-z|$. Hence $N_e(x, \kappa R)\subset B$ and
$$|D|\geq |N_e(x, \kappa R)|\geq \mu |B_{\kappa R}(x)|=\mu \kappa^n |B|.$$
This completes the proof of \eqref{CS-cond}.

The coercivity estimate \eqref{coer} is now a consequence of \eqref{CS-cond} and \cite[Theorem~1.2]{CS20}. Moreover, by \cite[Theorem~1.3]{CS20}, we have
\begin{equation}\label{P6}
\int_{B_1}\int_{B_1} \frac{(w(x)-w(y))^2}{|x-y|^{n+2s}}\dx\, \dy\leq C \int_{B_2}\int_{B_2} (w(x)-w(y))^2 K(x, y)\dx\, \dy
\end{equation}
for some universal constant $C$. Now for any $\rho>0$ and $x_0\in\Rn$, we define $K_\rho(x, y)=\rho^{n+2s} K(x_0+\rho x, x_0+\rho y)$. Then $K_\rho$ satisfies
\eqref{new-lower} and \eqref{new-upper} with the same constants $\lambda$ and  $\Lambda_p$. Hence \eqref{P6} holds with respect to the kernel $K_\rho$, giving us
$$
\int_{B_\rho(x_0)}\int_{B_\rho(x_0)} \frac{(w(x)-w(y))^2}{|x-y|^{n+2s}}\dx\, \dy\leq C \int_{B_{2\rho}(x_0)}\int_{B_{2\rho}(x_0)} (w(x)-w(y))^2 K(x, y)\dx\, \dy
$$
for all $\rho>0$ and $x_0\in\Rn$. Now \eqref{ball-coer} can be obtained using the above estimate and  a covering argument similar to  Dyda-Kassmann \cite[Lemma~6.13]{DK20}.
This completes the proof.
\end{proof}

Now we can complete the proof of Theorem~\ref{main}.
\begin{proof}[Proof of Theorem~\ref{main}]
Since the coercivity estimate is already established by Lemma~\ref{L2.2}, we only argue the regularity estimate. This will be done by verifying the conditions of \cite[Theorem~1.7]{DK20} and the proof will follow from
\cite{DK20}.

From \eqref{new-upper} we see that
$$\int_{B^c_r(x)} K(x, y)\dy=\sum_{j=0}^\infty \int_{A_{2^{j}r}(x)} K(x, y)\dy\leq \Lambda_p \sum_{j=0}^\infty (2^{j}r)^{-n-2s} |A_{2^{j}r}(x)|\leq \Lambda_p C_n r^{-2s} \sum_{j=0}^\infty (4^{-s})^j $$
for all $r>0$. Thus condition (D) of \cite{DK20} holds.

Let $w\in H^s(B_1)$ and $E:H^s(B_1)\to H^s(\Rn)$ be the extension operator with $\norm{E}\leq c$. 
As we are interested in computing seminorm, after substracting a constant we may assume that $\int_{B_1}w=0$.
This helps us to apply fractional Poincar\'e ineqality to conclude
$$\norm{Ew}_{\dot{H}^s(\Rn)}\leq C_1\norm{w}_{\dot{H}^s(B_1)}.$$
Applying \cite[Theorem~4.1]{IS20} we get for some constant $C$, dependent on $n, s, \Lambda_p$, that
\begin{align*}
 \int_{B_1}\int_{B_1} (w(x)-w(y))^2 K(x, y)\dx\, \dy &\leq  \int_{\Rn}\int_{\Rn} (Ew(x)-Ew(y))^2 K(x, y)\dx\, \dy
 \\
 &\leq C \int_{\Rn}\int_{\Rn} \frac{(Ew(x)-Ew(y))^2}{|x-y|^{n+2s}}\dx\, \dy
 \\
 &\leq C C_1^2 \int_{B_1}\int_{B_1} \frac{(w(x)-w(y))^2}{|x-y|^{n+2s}}\dx\, \dy.
\end{align*}
Using a standard scaling argument, as in Lemma~\ref{L2.2}, we then have, for all $w\in H^s(B_\rho(x_0))$ that
\begin{equation}\label{P7}
A^{-1} \norm{w}^2_{\dot{H}^s(B_\rho(x_0))}\leq \int_{B_{\rho}(x_0)}\int_{B_{\rho}(x_0)} (w(x)-w(y))^2 K(x, y)\dx\, \dy\leq A \norm{w}^2_{\dot{H}^s(B_\rho(x_0))}
\end{equation}
for some constant $A$, dependent on $n, s, p, \lambda, \Lambda_p$. This gives us \cite[Condition (A)]{DK20}.

To verify condition (B) of \cite{DK20} one can consider, for $\rho>0$, any Lipschitz cut-off function $\tau$ satisfying $|\tau(x)-\tau(y)|\leq (C\rho^{-1}|x-y|)\wedge 1$, for some universal constant $C$. Then, in view of the
calculations in \eqref{P2}, it can be easily checked that
$$\int_{\Rn}(\tau(x)-\tau(y))^2K(x, y)\dy\leq C^2\rho^{-2}\int_{B_\rho(x)}|x-y|^2K(x, y)\dy + \int_{B^c_\rho(x)}K(x, y)\dy\leq \tilde{C}\rho^{-2s}.$$

Therefore, to conclude the claim of Theorem~\ref{main} using \cite[Theorem~1.7]{DK20}, we only need to check that the test functions $\varphi$ in \eqref{eqnweak} can be chosen from the space
\begin{align*}
\mathbb{V}_0(B_2):=\Bigl\{\varphi:\Rn\to \R\: &:\: \varphi\Big|_{\Omega^c}=0\; \text{and}\; \norm{\varphi}_{L^2(\Rn)}+\iint_{\Rn\times\Rn}(\varphi(x)-\varphi(y))^2K(x, y)\dx\dy<\infty
\\
&\qquad \text{for some $\Omega\Subset B_2$}\Bigr\}.
\end{align*}
This actually follows from the coercivity and approximation argument. Indeed, for any $\varphi\in \mathbb{V}_0$, we have $\varphi\in H^s(\Rn)$ from \eqref{coer}, and therefore, can be approximated using
functions in $C^\infty_c(B_2)$ and in $H^s(\Rn)$ norm. Applying Cauchy–Schwarz inequality and \eqref{P7} we easily obtain \eqref{eqnweak} for the function $\varphi\in \mathbb{V}_0$. This completes the proof.
\end{proof}

\bigskip
\subsection*{Acknowledgement}
A. Biswas was partially supported by 
an ANRF-ARG grant ANRF/ARG/ 2025/000019/MS .

\subsection*{Data Availability}  All data generated or analyzed during this study are included in this article.
\medskip

\noindent{\bf Declarations}
\subsection*{Conflicts of Interest}  The author declares that they have no conflict of interest to this work.
\subsection*{Use of AI} The author acknowledges the use of an AI tool (Claude Fable 5.1) in building Example 2.1 and
formulating the intuition for Lemma 2.1. The author has thoroughly verified all mathematical proofs in this work.

\end{document}